\documentclass[10pt]{amsart}
\usepackage[english]{babel}
\usepackage{amssymb}
\usepackage{amsmath}
\usepackage{srcltx}
\usepackage{lscape}

\newtheorem{dummy}{Dummy}

\newtheorem{lemma}[dummy]{Lemma}
\newtheorem{theorem}[dummy]{Theorem}

\newtheorem{corollary}[dummy]{Corollary}

\theoremstyle{definition}

\newtheorem{remark}[dummy]{Remark}

\newcommand{\ignore}[1]{}

\author{A. Owen}
\author{S. Pumpl\"un}

\email{adam.owen@nottingham.ac.uk; susanne.pumpluen@nottingham.ac.uk}
\address{School of Mathematical Sciences\\
University of Nottingham\\ University Park\\ Nottingham NG7 2RD\\
United Kingdom }



\begin{document}
\title[A generalisation of Amitsur's A-polynomials]
{A generalisation of Amitsur's A-polynomials}

\begin{abstract}
We find examples of polynomials $f\in D[t;\sigma,\delta]$ whose eigenring $\mathcal{E}(f)$ is a central simple algebra over the field $F = C \cap {\rm Fix}(\sigma) \cap {\rm Const}(\delta)$.
\end{abstract}

\maketitle

\section*{Introduction}

Let $K$ be a field of characteristic $0$ and $R=K[t;\delta]$ be the ring of differential polynomials with coefficients in $K$.
In order to derive results on the structure of the left $R$-modules $R/Rf$, Amitsur studied spaces of linear differential operators via differential transformations \cite{Am,Am2,Am3}.
He observed that every central simple algebra  $B$ over a field $F$ of characteristic 0 that is split by an algebraically closed field extension $K$ of $F$, is isomorphic to the eigenspace of some polynomial $f \in K[t;\delta]$, for a suitable derivation $\delta$ of $K$.
This identification of a central simple algebra $B$ with a suitable differential polynomial $f\in K[t;\delta]$ he called  {\it A-polynomial} also holds when $K$ has  prime characteristic $p$ \cite[Section 10]{Am}, \cite{P18}.

Let $D$ be a central division algebra  of degree $d$ over $C$, $\sigma$ an endomorphism of $D$ and $\delta$ a left $\sigma$-derivation of $D$. Our aim is to provide a partial answer to the following generalisation of Amitsur's investigation:
\begin{center}
{\it ``For which polynomials $f$ in a skew polynomial ring $D[t;\sigma,\delta]$ is the eigenring $\mathcal{E}(f)$ a central simple algebra over its subfield $F = C \cap {\rm Fix}(\sigma) \cap {\rm Const}(\delta)$?''}
\end{center}

After the preliminaries in Section 1, we investigate two different setups, always assuming that $f$ has degree $m\geq 1$ and that the minimal left divisor of $f$ is square-free. We look at generalised A-polynomials in $D[t;\sigma]$ in Section 2,
where  $\sigma$ is an automorphism of $D$ with $\sigma^n = \iota_u$ for some $u \in D^{\times}$. Then
$f$ is a generalised A-polynomial in $R$ if and only if $f$ right divides $u^{-1}t^n-a$ for some $a \in F$ (Theorem \ref{Sec 3: Theorem 3}).
If $n$ is prime and not equal to $d$, then $f$ is a generalised A-polynomial in $R$ if and only if one of the following holds:
(i) There exists some $a \in F^{\times}$ such that $ua \neq \prod\limits_{j=1}^n\sigma^{n-j}(b)$ for every $b \in D$, and $f(t)=t^n-ua.$ In this case $f$ is an irreducible polynomial in $R$.
(ii) $m \leq n$ and there exist  $c_1,c_2,\dots,c_{m-1},c_m,b \in D^{\times}$, $c_m=1$, such that  $u^{-1}\prod\limits_{j=0}^{n-1} \sigma^{n-j}(b)\in F^{\times}$, and $f(t) = \prod\limits_{i=1}^m (t-\Omega_{c_i}(b)).$   (Theorem \ref{Sec 3: Theorem 4}). In particular,
 $f$ is a generalised A-polynomial in $R=K[t;\sigma]$, $K$ a field, if and only if $f$ right divides $t^n-a$ in $R$ (Theorem \ref{Sec 3: Theorem 5}).
If moreover $n$ is prime then $f$ is a generalised A-polynomial in $R=K[t;\sigma]$ if and only if one of the following holds:
(i) There exists some $a \in F^{\times}$ such that $a \neq N_{K/F}(b)$ for any $b \in K$, and $f(t)=t^n-a.$ In this case $f$ is  irreducible.
(ii) $m \leq n$ and there exist some constants $c_1,c_2,\dots,c_{m-1},c_m,b \in K^{\times}$ with $c_m=1$, such that
$f(t) = \prod\limits_{i=1}^m (t-\Omega_{c_i}(b))$ (Corollary \ref{Sec 3: Theorem 6}).

In Section 3, we study generalised A-polynomials in $D[t;\delta]$, where $C$ has prime characteristic $p$ and 
 $\delta$ is an algebraic derivation of $D$ with minimum polynomial
$g(t)\in F[t]$ of degree $p^e$  such that $g(\delta) = \delta_c $ for some nonzero $c \in D$. Then
$f$ is a generalised A-polynomial in $D[t;\delta]$ if and only if $f$ right divides $g(t)-(b+c)$ for some $b \in F$. In particular, ${\rm deg}(f) \leq p^e$ (Theorem \ref{Sec 4: Theorem 4}).
 In the special case that $g(t)=t^p-at$, $f$ is a generalised A-polynomial in $R$ if and only if one of the following holds:
(i) $f(t) = h(t) = t^p - at - (b + c)$, and $V_p(\alpha) - a\alpha - (b + c) \neq 0$ for all $\alpha \in D$. In this case $f$ is irreducible in $R$.
(ii) $h(t) = t^p -at - (b + c)$ for some $a,b \in F$, $m \leq p$ and $f(t) = \prod\limits_{i=1}^m (t - \Omega_{c_i}(\alpha))$ for some $c_1,c_2,\dots,c_m,\alpha \in D^{\times}$ with $c_m = 1$, such that $V_p(\alpha) - a\alpha - (b + c) = 0$ (Theorem \ref{Sec 4: Theorem 5}).

\section{Preliminaries} \label{sec:prel}


\subsection{Skew Polynomial Rings} \label{subsec:SPRs} (\cite{MRS01, J96, GW04, Ore})


Let $D$ be a unital associative division algebra over its center $C$,  $\sigma$  an endomorphism of $D$, and  $\delta$  a left $\sigma$-derivation of $D$, i.e. $\delta$ is an additive map on $D$ satisfying $\delta(xy) = \sigma(x)\delta(y) + \delta(x)y$ for all $x,y \in D$. 
For $u \in D^{\times}$, $\iota_u(a)= uau^{-1}$ is called an \emph{inner automorphism} of $D$. If there exists $n \in \mathbb{Z}^+$ such that $\sigma^n = \iota_u$ for some $u \in D^{\times}$, and $\sigma^i$ is a not an inner derivation for $1 \leq i < n$, then $\sigma$ is said to have \emph{finite inner order} $n$. 
For $c \in D$, the derivation $\delta_c(a) = [c,a] = ca - ac$ for all $a \in D$ is called an {\it inner derivation}.  The {\it skew polynomial ring} $R=D[t;\sigma,\delta]$ is the set of skew polynomials $a_mt^m + a_{m-1}t^{m-1} + \dots + a_1t + a_0$ with $a_i \in D$, endowed with term-wise addition and multiplication defined by $ta = \sigma(a)t + \delta(a)$ for all $a \in D$. $R$ is a unital associative ring. If $\delta = 0$, we write $R = D[t;\sigma]$. If $\sigma = {\rm  id}_D$, we write $R = D[t;\delta]$.

For $f(t) = a_mt^m + a_{m-1}t^{m-1} + \dots + a_1t + a_0$ with $a_m \neq 0$, the {\it degree} of $f$, denoted by ${\rm deg}(f)$, is $m$, and by convention ${\rm deg}(0) = - \infty$. If $a_m = 1$, we call $f$ {\it monic}. We have ${\rm deg}(fg) = {\rm deg}(f) + {\rm deg}(g)$ and ${\rm deg}(f+g) \leq {\rm max}({\rm deg}(f),{\rm deg}(g))$ for all $f,g \in R$. A polynomial $f \in R$ is called {\it reducible} if $f=gh$ for some $g,h \in R$ such that ${\rm deg}(g),{\rm deg}(h) < {\rm deg}(f)$, otherwise we call $f$ {\it irreducible}. A polynomial $f \in R$ is called {\it right (resp. left) invariant} if $fR \subseteq Rf$ (resp. $Rf \subseteq fR$), i.e. $Rf$ (resp. $fR$) is a two-sided ideal of $R$. We call $f$ {\it invariant} if it is both right and left invariant.

$R$ is a left principal ideal domain. The {\it left idealiser}  $\mathcal{I}(f) = \lbrace g \in R : fg \in Rf \rbrace$  of  $f\in R$  is the largest subring of $R$ within which $Rf$ is a two-sided ideal. We define the {\it eigenring} of $f$ as $\mathcal{E}(f) =\mathcal{I}(f)/Rf= \lbrace g \in R : {\rm deg}(g) < m \text{ and } fg \in Rf \rbrace.$ A nonzero $f \in R$ is said to be {\it bounded} if there exists another nonzero skew polynomial $f^{\star} \in R$, called a {\it bound} of $f$, such that $Rf^{\star}$ is the unique largest two-sided ideal of $R$ contained in the left ideal $Rf$. Equivalently, a nonzero polynomial in $f \in R$ is said to be bounded, if there exists a right invariant polynomial $f^{\star} \in R$, which is called a bound of $f$, such that $Rf^{\star} = {\rm Ann}_R(R/Rf) \neq \{ 0 \}.$ The annihilator ${\rm Ann}_R(R/Rf)$ of the left $R$-module $R/Rf$ is a two-sided ideal of $R$. When  $f$ is bounded and of positive degree, the nontrivial zero divisors in the eigenspace of $f$ are in one-to-one correspondence with proper right factors of $f$ in $R$:
 If $f$ is bounded and $\sigma\in {\rm Aut}(D)$, then $f$ is irreducible if and only if $\mathcal{E}(f)$ has no non-trivial zero divisors.
 Each non-trivial zero divisor $q$ of $f$ in $\mathcal{E}(f)$ gives a proper factor ${\rm gcrd}(q,f)$ of $f$. 
 \cite[Lemma 3, Proposition 4]{GLN13}.

If $D$ has finite dimension as an algebra over its center $C$, then $R=D[t;\sigma,\delta]$ is either a twisted polynomial ring or a differential polynomial ring \cite[Theorem 1.1.21]{J96}.

\subsection{Generalized A-polynomials}

Unless stated otherwise, from now on let $D$ be a unital associative division ring with center $C$, $\sigma\in {\rm End}(D)$, $\delta$  a left $\sigma$-derivation of $D$, and let $F=C \cap {\rm Fix}(\sigma) \cap {\rm Const}(\delta)$. We are interested in the question:
\begin{center}
``For $f \in R = D[t;\sigma,\delta]$ when is $\mathcal{E}(f)$ a central simple algebra over the field $F$?''
\end{center}

 We call $f\in R$ a {\it generalised A-polynomial} if $\mathcal{E}(f)$ is a central simple algebra over $F$.
For each $v \in D^{\times}$, we define a map $\Omega_v: D \longrightarrow D$ by
$\Omega_v(\alpha) = \sigma(v)\alpha v^{-1} + \delta(v)v^{-1}.$

\begin{lemma}\cite[Lemma 2 for $\sigma = {\rm id}$]{Am}\label{Sec 2: Lemma}
Let $\alpha,\beta \in D$. Then $(t - \alpha) \sim (t - \beta)$ in $D[t;\sigma,\delta]$ if and only if  $\Omega_v(\alpha) = \beta$ for some $v \in D^{\times}$.
\end{lemma}
\begin{proof}
$(t-\alpha) \sim (t-\beta)$ is equivalent to the existence of $v,w \in D^{\times}$ such that $w(t-\alpha)=(t-\beta)v$ \cite[pg.~33]{jacobson1943theory}, i.e. there exists $v,w \in D^{\times}$ such that $w(t-\alpha) = \sigma(v)t + \delta(v) - \beta v.$
This is the case if and only if $w=\sigma(v)$ and $w\alpha = \sigma(v)\alpha = \beta v - \delta(v)$. The result follows immediately.
\end{proof}


\section{Generalised A-polynomials in $D[t;\sigma]$}


Let $D$ be a central division algebra over $C$ of degree $d$ and $\sigma$ an automorphism of $D$ of finite inner order $n$, with $\sigma^n = \iota_u$ for some $u \in D^{\times}$.
Let $R=D[t;\sigma]$. Then $R$ has center $F[u^{-1}t^n]\cong F[x]$. For the remainder of this section we suppose that $f \in R$ is a monic polynomial of degree $m\geq 1$ such that the greatest common right divisor of $f$ and $t$ (denoted $(f,t)_r$) is one. Then $f^*\in C(R)$
 \cite[Lemma 2.11]{GLN13}.
We define the \emph{minimal central left multiple of $f$ in $R$} to be the unique  polynomial of minimal degree $h \in C(R)=F[u^{-1}t^n]$ such that $h = gf$ for some $g \in R$, and such that $h(t)=\hat{h}(u^{-1}t^n)$ for some monic $\hat{h}(x) \in F[x]$. Define $E_{\hat{h}}=F[x]/(\hat{h}(x))$.
Since $(f,t)_r=1$, $h(t)$ is a bound of $f$.
 $E_{\hat{h}} = F[x]/(\hat{h}(x))$ is a field if and only if $\hat{h}(x)\in F[x]$ is irreducible.

Since $F[x]$ is a unique factorisation domain, we have
$\hat{h}(x)= \hat{\pi}_1^{e_1}(x)\hat{\pi}_2(x)^{e_2}\cdots \hat{\pi}_z(x)^{e_z}$ for some irreducible polynomials $\hat{\pi}_1,\hat{\pi}_2,\dots, \hat{\pi}_z \in F[x]$ such that $\hat{\pi}_i \neq \hat{\pi}_j$ for $i \neq j$, and some exponents $e_1,e_2,\dots,e_z \in \mathbb{N}$.
Henceforth we assume that $e_1=e_2=\cdots=e_z=1$, i.e. that $\hat{h}$ is square-free.
By the Chinese Remainder Theorem for commutative rings \cite[\textsection{5}]{Co63}
$E_{\hat{h}} \cong  E_{\hat{\pi}_1} \oplus E_{\hat{\pi}_2} \oplus \cdots \oplus E_{\hat{\pi}_z},$
where $E_{\hat{\pi}_i} = F[x]/(\hat{\pi}_i(x))$ for each $i$.
$\mathcal{E}(f)$ is a semisimple algebra over its center $E_{\hat{h}}$ \cite{AO}.
Thus $\mathcal{E}(f)$ has center $F$ if and only if
$z=1$ and $E_{\hat{\pi}_1}=F$, i.e. if and only if $\hat{h}$ is a degree 1  polynomial in $F[x]$. Hence under the global assumption that $\hat{h}$ is square-free, we see that for $f$ to be a generalised A-polynomial it is necessary that $\hat{h}$ be irreducible.
So assume that $\hat{h}$ is irreducible. Then  the eigenspace of $f$ is a central simple algebra over the field $E_{\hat{h}}$:

\begin{theorem}\cite{AO}\label{Sec 3: Theorem 2}
Suppose that $\hat{h}(x)$ is irreducible in $F[x]$. Then $f=f_1f_2\cdots f_l$ where $f_1,f_2,\dots,f_l$  are irreducible polynomials in $R$ such that $f_i \sim f_j$ for all $i,j$. Moreover,
 $$ \mathcal{E}(f) \cong M_l(\mathcal{E}(f_i))$$ is a central simple algebra of degree $s=\frac{ldn}{k}$ over the field $E_{\hat{h}}$ where $k$ is the number of irreducible factors of $h(t)\in R$.
In particular, $deg(\hat{h}) = {\rm deg}(h)/n = \frac{dm}{s}$ and $[\mathcal{E}(f) :F] =  mds.$
\end{theorem}

\begin{theorem}\label{Sec 3: Theorem 3}
Suppose that $\hat{h}(x)$ is irreducible in $F[x]$. Then
$f$ is a generalised A-polynomial in $R$ if and only if $\hat{h}(x) = x - a$ for some $a \in F$ if and only if $f$ right divides $u^{-1}t^n-a$ for some $a \in F$. In particular, if $f$ is a generalised A-polynomial, then $m \leq n$.
\end{theorem}
\begin{proof}
Suppose that $f$ is a generalised A-polynomial in $R$. By the paragraph preceding Theorem \ref{Sec 3: Theorem 2}, for $f$ to be a generalised  A-polynomial it is necessary that $\hat{h}(x)=x-a$ for some $a \in F$. Conversely if $\hat{h}(x)=x-a \in F[x]$, then $E_{\hat{h}} = F[x]/(x-a) = F$. Hence $\mathcal{E}(f)$ is a central simple algebra over $F$ by Theorem \ref{Sec 3: Theorem 2}, i.e. $f$ is a generalised A-polynomial. It is easy to see that $\hat{h}(x) = x-a$ is equivalent to $f$ being a right divisor of $u^{-1}t^n-a$ by definition of the minimal central left multiple. Moreover, if $f$ right divides $u^{-1}t^n-a$, then ${\rm deg}(f) \leq   n$.
\end{proof}

For $n$ prime we are able to provide a more concrete description of $f$:

\begin{theorem}\label{Sec 3: Theorem 4}
Suppose that $\hat{h}(x)$ is irreducible in $F[x]$.
Suppose that $n$ is prime and not equal to $d$. Then $f$ is a generalised A-polynomial in $R$ if and only if one of the following holds:
\begin{enumerate}
\item There exists some $a \in F^{\times}$ such that $ua \neq \prod\limits_{j=1}^n\sigma^{n-j}(b)$ for every $b \in D$, and $f(t)=t^n-ua.$ In this case $f$ is an irreducible polynomial in $R$.
\item  $m \leq n$ and there exist  $c_1,c_2,\dots,c_{m-1},c_m,b \in D^{\times}$, $c_m=1$, such that  $u^{-1}\prod\limits_{j=0}^{n-1} \sigma^{n-j}(b)\in F^{\times}$, and
    $f(t) = \prod\limits_{i=1}^m (t-\Omega_{c_i}(b)).$
\end{enumerate}
\end{theorem}

\begin{proof}
By Theorem \ref{Sec 3: Theorem 3}, $f$ is a generalised A-polynomial in $R$ if and only if $f$ right divides $u^{-1}t^n-a$ for some $a \in F^{\times}$.
So suppose that $f$ is a generalised A-polynomial in $R$, then there exists some $a \in F^{\times}$ and some  nonzero $g \in R$ such that
\begin{equation}\label{f right divides u^-1t^n-a}
u^{-1}t^n-a = gf.
\end{equation}
In the notation of Theorem \ref{Sec 3: Theorem 2}, $ldn = ks$ and since $f$ is a generalised A-polynomial ${\rm deg}(\hat{h})=\frac{dm}{s} =1$, i.e.  $dm=s$. Combining these yields $\frac{n}{k} = \frac{m}{l} \in \mathbb{N}$. That is $k$ must divide $n$, and so we must have that $k=1$ or $k=n$ as $n$ is prime. We analyse the cases $k=1$ and $k=n$ separately.
\\
First suppose that $k=1$, then $h(t)$ is irreducible. Therefore Equation (\ref{f right divides u^-1t^n-a}) becomes
$u^{-1}t^n-a = gf(t)$
for some $a \in F^{\times}$ and some $g \in D^{\times}$. This yields $g=u^{-1}$ and $f(t) = t^n-ua$ for some $a \in F^{\times}$. Suppose that $f$ were reducible, then $f$ would be the product of $n$ linear factors as $n$ is prime, hence $f$ is irreducible if and only if $ua \neq \prod\limits_{j=1}^n \sigma^{n-j}(b)$ for any $b \in D$, by \cite[Corollary 3.4]{brown2018}.

On the other hand, if $k=n$, then $h(t)$ is equal to a product of $n$ linear factors in $R$, all of which are  similar. Also, since $\frac{n}{k}=\frac{m}{l}$ and $n=k$, we have $m=l \leq n$. Hence $f$ is the product of $m \leq n$ linear factors in $R$, all of which are similar to each other.

So there exist constants $b_1,b_2,\dots,b_m \in D^{\times}$ such that $(t-b_i) \sim (t-b_j)$ for all $i,j \in \{1,2,\dots,m\}$, and
$f(t) = \prod\limits_{i=1}^m (t-b_i).$
 In particular $(t-b_i) \sim (t-b_m)$ for all $i\neq m$, which is true if and only if there exist constants $c_1,c_2,\dots,c_{m-1},c_m \in D^{\times}$ such that $b_i = \Omega_{c_i}(b_m)$ for all $i$ by Lemma \ref{Sec 2: Lemma}. Hence setting $b=b_m$ and $c_m=1$ yields $f(t) = \prod\limits_{i=1}^m (t-\Omega_{c_i}(b)).$ Finally, we note that $(t-b)\vert_r (t^n - ua)$ for some $a \in F^{\times}$ if and only if $u^{-1}\prod\limits_{j=0}^{n-1} \sigma^{n-j}(b) = a \in F^{\times}$, by \cite[Corollary 3.4]{brown2018}.
\end{proof}

If $e_i > 1$ for at least one $i$, then it is not clear to the authors when $\mathcal{E}(f)$ is a central simple algebra over the field $F$.

\subsection{Generalised A-polynomials in $K[t;\sigma]$}

Throughout this section we suppose that $R=K[t;\sigma]$ with $K$ a field, and that $\sigma$ is an automorphism of $K$ of finite order $n$ with fixed field $F$. Now the center of $R$ is $F[t^n] \cong F[x]$.
Let $f \in R$ be of degree $m \geq 1$ and satisfy $(f,t)_r=1$, and suppose that $f$ has minimal central left multiple $h(t)=\hat{h}(t^n)$, $\hat{h} \in F[x]$ an irreducible monic polynomial. Again, we consider only those $f\in R$ where $\hat{h}\in F[x]$ is square-free.

\begin{theorem}\label{Sec 3: Theorem 5}
 $f$ is a generalised A-polynomial in $R$ if and only if $\hat{h}(x) = x - a$ for some $a \in F[x]$ if and only if $f$ right divides $t^n-a$ in $R$.
\end{theorem}

This follows from Theorem \ref{Sec 3: Theorem 3}. If $n$ is prime, then the following is an immediate corollary to both Theorem \ref{Sec 3: Theorem 3} and Theorem \ref{Sec 3: Theorem 4}:

\begin{corollary}\label{Sec 3: Theorem 6}
Let $n$ be prime. Then $f$ is a generalised A-polynomial in $R$ if and only if one of the following holds:
\begin{enumerate}
\item There exists some $a \in F^{\times}$ such that $a \neq N_{K/F}(b)$ for any $b \in K$, and $f(t)=t^n-a.$ In this case $f$ is an irreducible polynomial in $R$.
\item  $m \leq n$ and there exist some constants $c_1,c_2,\dots,c_{m-1},c_m,b \in K^{\times}$ with $c_m=1$, such that 
$f(t) = \prod\limits_{i=1}^m (t-\Omega_{c_i}(b)).$
\end{enumerate}
\end{corollary}

\begin{proof}
The proof is identical to the proof of Theorem \ref{Sec 3: Theorem 4} with $d=u=1$. The condition that $\prod\limits_{j=0}^{n-1}\sigma^j(b)$ lies in $F^{\times}$ is always satisfied, since $\prod\limits_{j=0}^{n-1}\sigma^j(b)=N_{K/F}(b)\in F^{\times}$ for all $b \neq 0$.
\end{proof}

In particular, let $K = \mathbb{F}_{q^n}$, where $q=p^e$ for some prime $p$ and exponent $e \geq 1$, and where $\sigma : K\longrightarrow K,$ $a \mapsto a^q$ is the Frobenius automorphism of order $n$, with fixed field $F=\mathbb{F}_q$.  Here the only central division algebra over $\mathbb{F}_q$ is $\mathbb{F}_q$ itself.
 The following result is therefore an easy consequence of Theorems \ref{Sec 3: Theorem 2} and \ref{Sec 3: Theorem 3}:

\begin{corollary}\label{Sec 3: Theorem 7}
Suppose that $f \in \mathbb{F}_{q^n}[t,\sigma]$ satisfies $(f,t)_r=1$, and has minimal central left multiple $h(t) = \hat{h}(t^n)$ for some irreducible polynomial $\hat{h} \in \mathbb{F}_q[x]$. Then $f$ is an A-polynomial if and only if $m \leq n$ and there exist some constants $c_1,c_2,\dots,c_{m-1},c_m,b \in \mathbb{F}_{q^n}^{\times}$ with $c_m=1$, such that
$f(t) = \prod\limits_{i=1}^m (t-\Omega_{c_i}(b)).$ In particular, $f$ is a reducible polynomial in $\mathbb{F}_{q^n}[t,\sigma]$, unless $m=1$.
\end{corollary}

The result follows identically to the $n=k$ case in the proof of Theorem \ref{Sec 3: Theorem 4}.

\section{Generalised A-polynomials in $D[t;\delta]$}

From now on let $R=D[t;\delta]$ where $D$ is a central division algebra of degree $d$ over $C$.
Assume that $C$ has prime characteristic $p$, and that $\delta$ is an algebraic derivation of $D$ with minimum polynomial
$g(t)=t^{p^e}+ \gamma_1t^{p^{e-1}} + \cdots + \gamma_et \in F[t],$
 such that $g(\delta)(a) = [c,a] = ca-ac$ for some nonzero $c \in D$ and for all $a \in D$. 
 Here, $F = C \cap {\rm Const}(\delta)$ ($D=K$ is a field is included here as special case).
Then $R$ has center $F[g(t)-c] \cong F[x]$.
 For every $f \in R$, the \emph{minimal central left multiple of $f$ in $R$} is the unique  polynomial of minimal degree $h \in C(R)=F[x]$ such that $h = gf$ for some $g \in R$, and such that $h(t)=\hat{h}(g(t)-c)$ for some monic $\hat{h}(x) \in F[x]$.
All $f \in R = D[t;\delta]$ have a unique minimal central left multiple, which is  a bound of $f$.

Again we can restrict our investigation to the case $\hat{h}$ is square-free in $F[x]$, and note that it is necessary that $\hat{h}$ be irreducible in $F[x]$ for $f$ to be a generalised A-polynomial in $R$.

\begin{theorem}\cite{AO}\label{Sec 4: Theorem 3}
Suppose that $\hat{h}(x)$ is irreducible in $F[x]$. Then $f=f_1f_2\cdots f_l$ where $f_1,f_2,\dots,f_l$  are irreducible polynomials in $R$ such that $f_i \sim f_j$ for all $i,j$. Moreover $$ \mathcal{E}(f) \cong M_l(\mathcal{E}(f_i))$$ is a central simple algebra of degree $s=\frac{ldp^e}{k}$ over the field $E_{\hat{h}}$ where $k$ is the number of irreducible factors of $h\in R$.
In particular ${\rm deg}(\hat{h}) = {\rm deg}(h)/p^e = \frac{dm}{s}$ and $[\mathcal{E}(f) :F] =  mds.$
\end{theorem}

We obtain the following:

\begin{theorem}\label{Sec 4: Theorem 4}
Suppose that $\hat{h}(x)$ is irreducible in $F[x]$. Then
$f$ is a generalised A-polynomial in $R$ if and only if $f$ right divides $g(t)-(b+c)$ for some $b \in F$. In particular, ${\rm deg}(f) \leq p^e$.
\end{theorem}
\begin{proof}
Suppose that $f$ is a generalised A-polynomial in $R$. For $f$ to be a generalised A-polynomial it is necessary that $\hat{h}(x)=x-b$ for some $b \in F$. Conversely if $\hat{h}(x)=x-b \in F[x]$, then $E_{\hat{h}} = F[x]/(x-b) = F$. Hence $\mathcal{E}(f)$ is a central simple algebra over $F$ by Theorem \ref{Sec 4: Theorem 3}, i.e. $f$ is a generalised A-polynomial. It is easy to see that $\hat{h}(x) = x-b$ is equivalent to $f$ being a right divisor of $g(t)-(b+c)$ by definition of the minimal central left multiple. Moreover, if $f$ right divides $g(t)-(b+c)$, then ${\rm deg}(f) \leq {\rm deg}(g(t)-(b+c)) = p^e$.
\end{proof}

In $D[t;\delta]$, we have
 $(t-b)^p = t^p - V_p(b),$ $ V_p(b) = b^p + \delta^{p-1}(b) + \nabla_b$
 for all $b \in D$, where $\nabla_b$ is a sum of commutators of $b,\delta(b),\delta^2(b),\dots,\delta^{p-2}(b)$ \cite[pg.~17-18]{J96}. In particular, if $D$ is commutative, then  $\nabla_b = 0$ and $V_p(b) = b^p + \delta^{p-1}(b)$ for all $b \in D$. Using the identities $t^p = (t-b)^p + V_p(b)$ and $t = (t-b) + b$ for all $b \in D$, we arrive at:

\begin{lemma}\cite[Proposition 1.3.25 ($e=1$)]{J96}\label{Sec 4: Lemma 1}
Let $f(t)=t^p-a_1t-a_0 \in D[t;\delta]$ and $b \in D$. Then $(t-b)\vert_r f(t)$ if and only if $V_p(b)-a_1b-a_0 = 0.$
\end{lemma}

If $e=1$ (i.e. $\delta$ is an algebraic derivation of $D$ of degree $p$), we can determine necessary and sufficient conditions for $f$ to be an A-polynomial in $R$:

\begin{theorem}\label{Sec 4: Theorem 5}
Let $\delta$ be an algebraic derivation of $D$ of degree $p$ with minimum polynomial $g(t)=t^p-at$ such that $g(\delta) = \delta_c$ for some $c \in D$. Suppose that $\hat{h}(x)$ is irreducible in $F[x]$.
Then $f$ is a generalised A-polynomial in $R$ if and only if one of the following holds:
\begin{enumerate}
\item $f(t) = h(t) = t^p - at - (b + c)$, and $V_p(\alpha) - a\alpha - (b + c) \neq 0$ for all $\alpha \in D$. In this case $f$ is irreducible in $R$.
\item $h(t) = t^p -at - (b + c)$ for some $a,b \in F$, $m \leq p$ and $f(t) = \prod\limits_{i=1}^m (t - \Omega_{c_i}(\alpha))$ for some $c_1,c_2,\dots,c_m,\alpha \in D^{\times}$ with $c_m = 1$ (w.l.o.g.) such that $V_p(\alpha) - a\alpha - (b + c) = 0.$
\end{enumerate}
\end{theorem}
\begin{proof}
By Theorem \ref{Sec 4: Theorem 4}, $f$ is a generalised A-polynomial in $R$ if and only if $f$ right divides $ t^p - at - (b + c)$ for some $b \in F$.
So suppose that $f$ is a generalised A-polynomial in $R$, then there exists some $b \in F$ and some  nonzero $f^{\prime} \in R$ such that
\begin{equation}\label{f right divides  t^p - at - (b + c)}
 t^p - at - (b + c) = f^{\prime}f
\end{equation}
In the notation of Theorem \ref{Sec 4: Theorem 3}, $ldp = ks$ and since $f$ is a generalised A-polynomial, ${\rm deg}(\hat{h})=\frac{dm}{s} =1$, i.e.  $dm=s$. Combining these yields $\frac{p}{k} = \frac{m}{l} \in \mathbb{N}$. That is $k$ must divide $p$, and so we must have that $k=1$ or $k=p$ as $p$ is prime.

First suppose that $k=1$, then $h(t)$ is irreducible in $R$. Therefore Equation (\ref{f right divides  t^p - at - (b + c)}) becomes
$t^p - at - (b + c) = f^{\prime}f$
for some $b \in F^{\times}$ and some $f^{\prime} \in D^{\times}$. This yields $f^{\prime}=1$ and $f(t) =  t^p - at - (b + c)$. Suppose that $f$ were reducible, then $f$ would be the product of $p$ linear factors as $p$ is prime, hence $f$ is irreducible if and only if $V_p(\alpha) - a\alpha - (b + c) \neq 0$ for any $\alpha \in D$, by Lemma \ref{Sec 4: Lemma 1}.

On the other hand, if $k=p$, then $h(t)$ is equal to a product of $p$ linear factors in $R$, all of which are similar to one another. Also, since $\frac{p}{k}=\frac{m}{l}$ and $p=k$, we have $m=l \leq p$. Hence $f$ is the product of $m \leq p$ linear factors in $R$, all of which are mutually similar to each other.

So there exist constants $\alpha_1,\alpha_2,\dots,\alpha_m \in D^{\times}$ such that
$f(t) = \prod\limits_{i=1}^m (t-\alpha_i),$
and $(t-\alpha_i) \sim (t-\alpha_j)$ for all $i,j \in \{1,2,\dots,m\}$.
 In particular $(t-\alpha_i) \sim (t-\alpha_m)$ for all $i\neq m$, which is true if and only if there exist constants $c_1,c_2,\dots,c_{m-1},c_m \in D^{\times}$ such that $\alpha_i = \Omega_{c_i}(\alpha_m)$ for all $i$ by  Lemma \ref{Sec 2: Lemma}. Hence setting $\alpha=\alpha_m$ and $c_m=1$ yields $f(t) = \prod\limits_{i=1}^m (t-\Omega_{c_i}(\alpha)).$ Finally, we note that $(t-\alpha)$ right divides $t^p-at-(b+c)$ if and only if $V_p(\alpha) - a\alpha - (b + c) = 0$ by Lemma \ref{Sec 4: Lemma 1}.
\end{proof}

\begin{remark}
Suppose on the other hand that $C$ has characteristic $0$ and $\delta$ is the inner derivation  $\delta_c$. Then $R$ has center $C[t-c] \cong C[x]$. i.e. $F = C$. In this case the A-polynomials are trivial: if $\hat{h}(x)$ is irreducible in $C[x]$ then
$f$ is a generalised A-polynomial in $R$ if and only if $f(t)=(t-c)+a$ for some $a \in C$.
In this case, $\mathcal{E}(f)=D$.
\end{remark}

\end{document}